\documentclass[11pt,oneside,letterpaper]{article}

\usepackage{mathtools}
\mathtoolsset{showonlyrefs=true}
\usepackage{amsthm, amsmath, amssymb, amsfonts, graphicx, epsfig}
\usepackage{accents}
\usepackage{bbm}
\usepackage{mathrsfs}

\usepackage{epstopdf}

\usepackage{graphicx}
\usepackage[font=sl,labelfont=bf]{caption}
\usepackage{subcaption}

\usepackage{float}
\restylefloat{table}

\usepackage{enumerate}

\usepackage[usenames,dvipsnames]{color}

\usepackage[colorlinks=true, pdfstartview=FitV, linkcolor=blue,
            citecolor=blue, urlcolor=blue]{hyperref}

\usepackage{url}
\makeatletter\def\url@leostyle{%
 \@ifundefined{selectfont}{\def\UrlFont{\sf}}{\def\UrlFont{\scriptsize\ttfamily}}} \makeatother

\usepackage[margin=1.1in, letterpaper]{geometry}

\newtheorem{theorem}{Theorem}
\newtheorem{proposition}[theorem]{Proposition}
\newtheorem{lemma}[theorem]{Lemma}
\newtheorem{corollary}[theorem]{Corollary}

\theoremstyle{definition}

\newtheorem{example}[theorem]{Example}

\newtheorem{remark}[theorem]{Remark}

\numberwithin{equation}{section}
\numberwithin{theorem}{section}

\definecolor{Red}{rgb}{0.9,0,0.0}
\definecolor{Blue}{rgb}{0,0.0,1.0}

\def\cA{\mathcal{A}}

\def\bN{\mathbb{N}}

\def\bP{\mathbb{P}}

\def\bR{\mathbb{R}}

\def\sF{\mathscr{F}}

\newcommand{\set}[1]{\{#1\}}            

\DeclareMathOperator*{\arginf}{arg\,inf}

\def\covdist{\;{\stackrel{d}{\longrightarrow}}\;}

\title{\vspace{-2em} Trajectory Fitting Estimators for SPDEs \\ Driven by Additive Noise}

\author{
    Igor Cialenco \\[-0.3ex]
    \url{cialenco@iit.edu}  \\[-0.9ex]
    \url{http://math.iit.edu/\~igor}
 \and
     Ruoting Gong \\[-0.3ex]
    \url{rgong2@iit.edu}  \\[-0.9ex]
    \url{http://mypages.iit.edu/\~rgong2}
 \and
     Yicong Huang \\[-0.3ex]
    \url{yhuang37@hawk.iit.edu}  \\[-0.9ex]
 \and \\
        {\footnotesize Department of Applied Mathematics, Illinois Institute of Technology} \\
        {\footnotesize W 32nd Str, John T. Rettaliata Engineering Center, Room 208, Chicago, IL 60616, USA}\\
        }

\date{ {\small  
First Circulated: June 15, 2016 \\
This Version: November 10, 2016 \\[1em]
Forthcoming in \textit{Statistical Inference for Stochastic Processes}
}}

\begin{document}

\maketitle

\vspace{-2em}

\smallskip
{\footnotesize
\begin{tabular}{l@{} p{350pt}}
  \hline \\[-.2em]
  \textsc{Abstract}: \ &
  In this paper we study the problem of estimating the drift/viscosity coefficient for a large class of linear, parabolic stochastic partial differential equations (SPDEs) driven by an additive space-time noise. We propose a new class of estimators, called trajectory fitting estimators (TFEs). The estimators are constructed by fitting the observed trajectory with an artificial one, and can be viewed as an analog to the  classical least squares estimators from the time-series analysis. As in the existing literature on statistical inference for SPDEs, we take a spectral approach, and assume that we observe the first $N$ Fourier modes of the solution, and we study the consistency and the asymptotic normality of the TFE, as $N\to\infty$.
  \\[0.5em]
\textsc{Keywords:} \ &  stochastic partial differential equations, trajectory fitting estimator, parameter estimation, inverse problems, estimation of viscosity coefficient \\
\textsc{MSC2010:} \ &   60H15, 35Q30, 65L09 \\[1em]
  \hline
\end{tabular}
}

\section{Introduction}
In this paper we study a parameter estimation problem for the drift/viscosity coefficient of a linear parabolic stochastic partial differential equation (SPDE) driven by an additive space-time noise (possibly colored in space). We assume that the  observable variable  is one path of the solution, observed continuously on a finite time-interval as an element of an infinite dimensional Hilbert space. More precisely, we assume that the first $N$ Fourier modes of the solution are observed continuously on the time-interval $[0,T]$, and we investigate the asymptotics of the proposed estimators, as $N\to\infty$. Most of the existing literature on statistical inference for SPDEs has a similar spectral approach, starting with the seminal paper by Huebner and Rozovskii~\cite{HuebnerRozovskii}, and essentially explore the Maximum Likelihood Estimator (MLE) approach. The main idea can be  summarized as follows: the measures generated by the solution $u$ of a parabolic SPDE that corresponds to different drift parameters are singular to each other (under some appropriate subordination assumption), which indicates that the true parameter can be found exactly. Indeed, by taking an appropriate finite dimensional projection of the solution, one obtains a finite dimensional system of  stochastic ordinary differential equations (SODEs) for which an $\textrm{MLE}$ for the drift coefficient exists. One can show that as the dimension $N$ of the projection increases, the $\textrm{MLE}$ converges to the true parameter and it is asymptotically normal. The MLE type estimators for SPDEs are well understood, to some extent, and for more details on this method we refer to the survey~\cite{Lototsky2009Survey} and to the monograph \cite{Bishwal2008} on linear SPDEs, to~\cite{IgorNathanAditiveNS2010} for an adaptation of this method to a nonlinear SPDE, to~\cite{CXu2014,CXu2015} for the hypothesis testing for stochastic fractional heat equation, and to~\cite{MishraBishwal1995,PiterbargRozovskii1997, Markussen2003, PrakasaRao2003, MishraPrakasaRao2002} for discrete time sampling. For linear, diagonalizable\footnote{A diagonalizable SPDE is an equation for which the first $N$ Fourier coefficients of the solution form an $N$-dimensional \textit{decoupled} system of ordinary stochastic differential equation. For a formal definition, in terms of the differential operators and the structure of the noise term, see for instance~\cite{Lototsky2009Survey}.} SPDEs, the $\textrm{MLE}$ can be computed explicitly, and of course, from a statistical point of view there is no need to study other type of estimators. In general, this statement does not hold true for non-diagonalizable equations such as nonlinear SPDEs or SPDEs driven by a multiplicative noise. While the parameter estimation problem for SODEs went way beyond the MLEs (cf. the monograph~\cite{KutoyantsBook2004}), there exist a limited number of works dedicated to the non-MLE statistical inference for SPDEs. For example, in~\cite{CialencoLototsky2009} the authors explore the singularity of corresponding probability measures and derive a closed-form estimators for the drift coefficient for some linear parabolic SPDEs driven by a multiplicative noise (of special structure). This study is the first attempt to investigate different type of estimators, and it is related to what is known in the literature the  trajectory fitting estimators (TFEs). Using as observations the first $N$ Fourier modes, we construct the TFE by analogy to the TFE for continuously observed finite dimensional ergodic diffusion processes first introduced by Y.~A.~Kutoyants \cite{Kutoyants1991}; see also \cite[Section 1.3 \& Section 2.3]{KutoyantsBook2004} and references therein. It would be fair to call this estimator Kutoyants Estimator, but it seems that TFE is already a well established name  and thus we will follow this terminology. The TFE can be also viewed as an analog of the least squares estimators widely used in time-series analysis.

As already mentioned, we study the asymptotic properties of the TFE as $N\to\infty$, in contrast to the diffusion processes where the asymptotics are done for the large-time regime. Surely one can investigate the large-time asymptotics for SPDEs too, but we find this case to be too similar to the estimators for diffusion processes and we omitted it here. In this study, we consider a fairly simple, although general, class of SPDEs: linear, parabolic, diagonalizable equations, driven by an additive space-time noise. The diagonalizable nature of these equations, allows to derive explicit expressions for the considered estimators and for the asymptotics of their first two moments, and hence to investigate the rate of convergence of these estimators. While simple, these equations can be viewed as a good approximation of some more complicated and practically important models. On the other hand, the obtained results will serve as benchmarks for future studies of more complicated and realistic models which will be addressed in the sequel. Under some general structural assumptions we prove consistency and asymptotic normality of the proposed estimators.

The paper is organized as follows. In Section~\ref{sec:Setup} we setup the problem and give some auxiliary results. The TFE is derived in Section~\ref{sec:TFE-construction}. Section~\ref{sec:mainResults} is devoted to the main results -- consistency of TFE (Theorem~\ref{thm:Consistency}) and asymptotic normality (Theorem~\ref{thm:AsymNormal}). In Section~\ref{sec:Examples} we discuss the TFE for two particular illustrative examples of SPDEs. Due to the nature of the proofs, some proofs required nontrivial and tedious computations for which, as appropriate, we  used symbolic computations in a mathematical software. All technical proofs and some auxiliary results are deferred to the Appendix.

\section{Setup of the problem and some auxiliary results}\label{sec:Setup}
Let $(\Omega,\sF,\set{\sF_t}_{t\geq 0},\bP)$ be a stochastic basis with usual assumptions, and let $H$ be a separable Hilbert space, with the corresponding inner product $(\,\cdot\,,\,\cdot\,)_{H}$. We consider the following evolution equation
\begin{align}\label{eq:mainAbstract}
du(t)+\left(\theta\cA_{1}+\cA_{0}\right)u(t)\,dt&=\sigma\,dW(t),
\end{align}
with initial condition $u(0)=u_{0}\in H$, and where $\cA_{0}$ and $\cA_{1}$ are linear operators on $H$, $W:=\{W(t)\}_{t\geq 0}$ is a cylindrical Brownian motion in $H$, and  $\sigma,\theta\in\bR_{+}:=(0,\infty)$. We will take the continuous-time observation framework by assuming that the solution $u$, as an object in $H$, or a finite dimensional projection in $H$ of it, is observed continuously in time for all $t\in[0,T]$, and for some fixed horizon $T$. In this framework, the parameter $\sigma$ can be found exactly, using a quadratic variation argument, and thus we will assume that $\sigma$ is known. We are interested in estimating the unknown parameter $\theta\in\Theta\subset\bR_{+}$.

We start with some structural assumptions on \eqref{eq:mainAbstract} related to the well-posedness of the solution. Throughout the paper, we assume that:
\begin{enumerate}[(i)]
\item The operators $\cA_{0}$ and $\cA_{1}$ have only point spectra, and a common system of eigenfunctions $\{h_{k}\}_{k\in\bN}$ that form a complete, orthonormal system in $H$. We denote the corresponding eigenvalues of $\cA_{0}$ and $\cA_{1}$ by $\rho_{k}$ and $\nu_{k}$, $k\in\bN$, respectively.
\item The sequence $\{\mu_{k}(\theta)\}_{k\in\bN}$, where $\mu_{k}(\theta):=\theta\nu_{k}+\rho_{k}$, is such that
    \begin{align}
    \lim_{k\rightarrow\infty}\mu_{k}(\theta)=+\infty,
    \end{align}
    where the convergence is uniform in $\theta\in\Theta$.
\item There exist universal constants $J\in\bN$ and $c_{0}>0$ such that, for any $k\geq J$ and any $\theta\in\Theta$,
    \begin{align}
    \frac{\nu_{k}}{\mu_{k}(\theta)}\leq c_{0}.
    \end{align}
\item The sequence $\{\nu_{k}\}_{k\in\bN}$ is such that $\lim_{k\rightarrow\infty}\nu_{k}=+\infty.$\footnote{Without loss of generality, we will assume that $\nu_{k}\geq 0$, for all $k\in\bN$.}
\item The noise term $W$ is a cylindrical Brownian motion in $H$, and has the following form
    \begin{align}\label{eq:W}
    W(t)=\sum_{k=1}^{\infty}\lambda_{k}^{-\gamma}h_{k}w_{k}(t),\quad t\geq 0,
    \end{align}
    for some $\gamma\geq 0$, where $\lambda_{k}:=\nu_{k}^{1/(2m)}$, $k\in\bN$, for some $m\geq 0$,\footnote{Of course, one can consider at once just $\lambda_{k}=\nu_{k}$. Our choice to consider $m$ is to put the results on par with the notations from the existing literature. As mentioned later, if $\cA_{0}$ and $\cA_{1}$ are some pseudo-differential operators, then it is convenient to denote by $2m$ the order of the leading order operator.} and where $w_{k}:=\{w_{k}(t)\}_{t\geq 0}$, $k\in\bN$, is a collection of independent standard Brownian motions.
\end{enumerate}

Conditions (i)--(v) imply that the equation \eqref{eq:mainAbstract} is linear, diagonalizable, parabolic, and that the solution exists and is unique; this can be established by standard methods from theory of SPDEs and we refer, for instance, to~\cite{Lototsky2009Survey, HuebnerLototskyRozovskii97, HuebnerRozovskii} for similar setup, or to~\cite{RozovskiiBook,ChowBook} for a general theory. Of course, one class of operators $\cA_{0}$ and $\cA_{1}$ that satisfy the above conditions are pseudo-differential operators on bounded domains with appropriate boundary conditions, with $\cA_{0}$ being subordinated to $\cA_{1}$ -- see Section \ref{sec:Examples} for some particular examples.

It is out of the scope of this work to go through the constructions of the scale of Sobolev space, or to investigate the exact order of regularity of the solution. We will only mention that generally speaking the series \eqref{eq:W} may diverge, but one can enlarge the underlying space $H$, such that $W$ is a square-integrable martingale. For example, let $\mathcal{X}$ be the closure of $H$ with respect to the norm
$\left\|f\right\|_{\mathcal{X}}:=\big(\sum_{k=1}^{\infty}k^{-2}\left(f,h_{k}\right)_{H}^{2}\big)^{1/2}$. Then, the unique solution to \eqref{eq:mainAbstract}, with initial condition $u(0)=u_{0}$, is a continuous (in mean-square sense) $\mathcal{X}$-valued stochastic process given by
\begin{align}
u(t)=\sum_{k=1}^{\infty}u_{k}(t)h_{k},\quad t\geq 0,
\end{align}
where, for each $k\in\mathbb{N}$, $u_{k}:=\{u_{k}(t)\}_{t\geq 0}$ satisfies the following ordinary stochastic differential equation (SDE)
\begin{align}\label{eq:FModes}
du_{k}(t)+ \mu_{k}(\theta)\,u_{k}(t)\,dt=\sigma\lambda_{k}^{-\gamma}\,dw_{k}(t),
\end{align}
with initial condition $u_{k}(0)=(u_{0},h_{k})_{H}$. The stochastic processes $u_{k}$, $k\in\mathbb{N}$, are the Fourier modes of the solution $u$ with respect to the basis $\{h_{k}\}_{k\in\mathbb{N}}$ of $H$, i.e., $u_{k}(t)=(u(t),h_{k})_{H}$, $t\geq 0$, $k\in\bN$. Note that the SDEs of the form \eqref{eq:FModes}, for $k\in\mathbb{N}$, provide an infinite system of decoupled/independent Ornstein--Uhlenbeck processes. As already mentioned, the decoupling nature of Fourier modes, or the diagonalizable property of the original equation, will play a critical role in our study, and it is essentially guaranteed by the assumptions (i) and (v). By It\^{o}'s formula, clearly we have
\begin{align}\label{eq:FModeSols}
u_{k}(t)=e^{-\mu_{k}(\theta)t}u_{k}(0)+\sigma\lambda_{k}^{-\gamma}e^{-\mu_{k}(\theta)t}\int_{0}^{t}e^{\mu_{k}(\theta)s}\,dw_{k}(s),\quad t\geq 0,\quad k\in\bN.
\end{align}

\subsection{Trajectory Fitting Estimators}\label{sec:TFE-construction}

The trajectory fitting estimators for continuous-time diffusion processes can be viewed as an analog of the least squares estimators widely used in time-series analysis. Following~\cite[Section 1.3 \& Section 2.3]{KutoyantsBook2004}, we will briefly describe the TFEs for finite-dimensional diffusions. Assume that the observed process $S(\theta):=\{S(t;\theta)\}_{t\geq 0}$ follows the dynamics
\begin{align}\label{eq:SDES}
dS(t;\theta)=b(\theta,S(t;\theta))dt+\sigma(S(t;\theta))\,dB(t),
\end{align}
where $B:=\{B(t)\}_{t\geq 0}$ is a one-dimensional standard Brownian motion, and $\theta$ is the parameter of interest. We assume that the drift $b$ and the volatility $\sigma$ are known, and that the solution to \eqref{eq:SDES} (with certain initial condition $S(0,\theta)=S_{0}$) exists and is unique, for any $\theta\in\Theta$. Let  $F:\bR\rightarrow\bR$ be a twice continuously differentiable function. By It\^{o}'s formula,
\begin{align}
F(S(t;\theta))&=F(S_{0})+\int_{0}^{t}\left(F'(S(s;\theta))b\left(\theta,S(s;\theta)\right)+\frac{1}{2}F''(S(s;\theta))\sigma^{2}(S(s;\theta))\right)ds\\
&\quad\,\,+\int_{0}^{t}F'(S(s;\theta))\sigma(S(s;\theta))\,dB(s).
\end{align}
For any $\theta\in\Theta$ and $t\in[0,T]$, let
\begin{align}
\widetilde{F}(t;\theta):=F(S_{0})+\int_{0}^{t}\left(F'(S(s;\theta))b\left(\theta,S(s;\theta)\right)+\frac{1}{2}F''(S(s;\theta))\sigma^{2}(S(s;\theta))\right)ds.
\end{align}
The \emph{trajectory fitting estimator}\footnote{The terminology comes from the fact that the estimator is obtained by fitting the observed trajectory with the artificial one.} $\widetilde{\theta}_{T}$ of $\theta$ is then defined as the solution to the minimization problem
\begin{align}\label{eq:TFESDE}
\widetilde{\theta}_{T}:=\arginf_{\theta\in\Theta}\int_{0}^{T}\left(F(S(t;\theta))-\widetilde{F}(t;\theta)\right)^{2}dt.
\end{align}
The choice of function $F$ depends on the underlying models, and has to be taken such that the estimator satisfies the desired asymptotic properties (consistency, asymptotic normality, etc).

In this study, it is enough to consider the quadratic function $F(x)=x^{2}$, the choice we make throughout. For each Fourier mode $u_{k}$, $k\in\bN$, by It\^{o}'s formula, we have
\begin{equation}\label{eq:uk2}
u_{k}^{2}(t)=u_{k}^{2}(0)+\int_{0}^{t}\left(\sigma^{2}\lambda_{k}^{-2\gamma}-2\mu_{k}(\theta)u_{k}^{2}(s)\right)ds+2\sigma\lambda_{k}^{-\gamma}\int_{0}^{t}u_{k}(s)\,dw_{k}(s),\quad t\geq 0.
\end{equation}
By \eqref{eq:TFESDE}, one can easily construct a TFE for $\theta$ based on the trajectory on $[0,T]$, for some fixed horizon $T>0$, of each Fourier mode $u_{k}$. The \emph{long-time} asymptotic behavior of such estimators as $T\rightarrow\infty$ has been well investigated (cf.~\cite{KutoyantsBook2004}), and thus we will omit it here.

By analogy to the construction of maximum likelihood estimators for SPDEs (cf.~\cite{Lototsky2009Survey}), we will construct a TFE for the unknown parameter $\theta$ based on the trajectories of the first $N$ Fourier modes of the solution. Moreover, for a fixed horizon $T>0$, we will study the \emph{large-space} asymptotic behavior of the TFE as the number of the Fourier modes increases, which is a distinguished feature for an infinite dimensional evolution system. Specifically, fix any $T>0$, and for any $\theta\in\Theta$, let
\begin{equation}\label{eq:uhat}
V_{k}(t;\theta):=u_{k}^{2}(0)+\int_{0}^{t}\left(\sigma^{2}\lambda_{k}^{-2\gamma}-2\mu_{k}(\theta)u_{k}^{2}(s)\right)ds,
\quad k\in\bN,\quad t\in [0,T].
\end{equation}
The TFE for the unknown parameter $\theta$ is defined as
\begin{equation}
\widetilde{\theta}_{N} =\widetilde{\theta}_{N}(\gamma,T,\sigma,m):=\arginf_{\theta\in\Theta}\sum_{k=1}^{N}\int_{0}^{T}\left(V_{k}(t;\theta)-u_{k}^{2}(t)\right)^{2}dt.
\end{equation}
We are interested in the asymptotic properties of $\widetilde{\theta}_{N}$, as $N\to\infty$, with $T$ being fixed.

In what follows, we will make use of the following notations. For any $t\in[0,T]$, let
\begin{align}\label{eq:notation}
\xi_{k}(t)&:=\!\int_{0}^{t}u_{k}^{2}(s)\,ds,\,\,\,\,X_{k}(t):=\!\int_{0}^{t}s\xi_{k}(s)\,ds
,\,\,\,\,Y_{k}(t):=\!\int_{0}^{t}\xi_{k}(s)\,ds,\,\,\,\,Z_{k}(t):=\!\int_{0}^{t}\xi_{k}^{2}(s)\,ds.
\end{align}
One advantage of the TFE is that it can be given by an explicit formula that does not contain a stochastic integral. Indeed, by \eqref{eq:uk2} and \eqref{eq:uhat},
\begin{equation}
\sum_{k=1}^{N}\int_{0}^{T}\left(V_{k}(t;\theta)-u_{k}^{2}(t)\right)^{2}dt=\sum_{k=1}^{N}\int_{0}^{T}\left(u_{k}^{2}(0)+\sigma^{2}\lambda_{k}^{-2\gamma}t-2\rho_{k}\xi_{k}(t)-u_{k}^{2}(t)-2\theta\nu_{k}\xi_{k}(t)\right)^{2}dt.
\end{equation}
The maximizer of the last expression, with respect to $\theta$, can be computed simply by finding the root of the first-order derivative. This yields the following explicit expression for the TFE
\begin{equation}\label{eq:thetaMain}
\widetilde{\theta}_{N}=-\frac{\sum_{k=1}^{N}\nu_{k}\left(\frac{1}{2}\xi_{k}^{2}(T)-u_{k}^{2}(0)Y_{k}(T)-\sigma^{2}\lambda_{k}^{-2\gamma}X_{k}(T)+2\rho_{k}Z_{k}(T)\right)}{2\sum_{k=1}^{N}\nu_{k}^{2}Z_{k}(T)}.
\end{equation}

\section{Main results}\label{sec:mainResults}
In what follows, we  will denote by $\theta$ the true parameter. For notational simplicity, the $T$ variable in $\xi_{k}(T)$, $X_{k}(T)$, $Y_{k}(T)$, $Z_{k}(T)$ and $A_{k}(T)$ will be omitted from now on. A straightforward algebraic computation leads to
\begin{align}\label{eq:thetaDif}
\widetilde{\theta}_{N}-\theta=-\frac{\sum_{k=1}^{N}\nu_{k}\left(\frac{1}{2}\xi_{k}^{2}-u_{k}^{2}(0)Y_{k}-\sigma^{2}\lambda_{k}^{-2\gamma}X_{k}+2\mu_{k}(\theta)Z_{k}\right)}{2\sum_{k=1}^{N}\nu_{k}^{2}Z_{k}}=:-\frac{\sum_{k=1}^{N}\nu_{k}A_{k}}{2\sum_{k=1}^{N}\nu_{k}^{2}Z_{k}}.
\end{align}

As usual, for two sequences of positive numbers $\{a_{n}\}_{n\in\bN}$ and $\{b_{n}\}_{n\in\bN}$, we will write $a_{n}\sim b_{n}$ if $\lim_{n\rightarrow\infty}a_{n}/b_{n}=1$, and will write $a_{n}\asymp b_{n}$ if there exist universal constants $K_{2}>K_{1}>0$, such that $K_{1}b_{n}\leq a_{n}\leq K_{2}b_{n}$ for $n\in\mathbb{N}$ large enough.

We start with a technical result regarding the leading order terms of the means and variances of $A_{k}$ and $Z_{k}$, as $k\to\infty$. The proof will be deferred to the Appendix.
\begin{proposition}\label{prop:asymp}
Let the assumptions \emph{(i) - (v)} be satisfied. Then, as $k\rightarrow\infty$,
\begin{align}\label{eq:ezk}
\mathbb{E}(Z_{k})&\asymp\frac{1}{\mu_{k}^{2}(\theta)}\left(u_{k}^{2}(0)+\sigma^{2}T\lambda_{k}^{-2\gamma}\right)^{2},\\
\label{eq:varzk} \emph{Var}(Z_{k})&\asymp\frac{\lambda_{k}^{-2\gamma}}{\mu_{k}^{5}(\theta)}\left(u_{k}^{2}(0)+\sigma^{2}T\lambda_{k}^{-2\gamma}\right)^{3},\\
\label{eq:eak} \mathbb{E}(A_{k})&\asymp\frac{\lambda_{k}^{-2\gamma}}{\mu_{k}^{2}(\theta)}\left(u_{k}^{2}(0)+\sigma^{2}T\lambda_{k}^{-2\gamma}\right),\\
\label{eq:varak}
\emph{Var}(A_{k})&\asymp\frac{\lambda_{k}^{-2\gamma}}{\mu_{k}^{3}(\theta)}\left(u_{k}^{2}(0)+\sigma^{2}T\lambda_{k}^{-2\gamma}\right)^{3}.
\end{align}
\end{proposition}

\begin{remark}\label{rem:asymp}
In the above proposition, the exact asymptotic behavior of the means and variances of $Z_{k}$ and $A_{k}$, as $k\rightarrow\infty$, can be obtained as follows:
\begin{align}\label{eq:ezkExact} \mathbb{E}(Z_{k})&\sim\frac{u_{k}^{4}(0)T}{4\mu_{k}^{2}(\theta)}+\frac{\sigma^{2}\lambda_{k}^{-2\gamma}u_{k}^{2}(0)T^{2}}{4\mu_{k}^{2}(\theta)}+\frac{\sigma^{4}\lambda_{k}^{-4\gamma}T^{3}}{12\mu_{k}^{2}(\theta)},\\
\label{eq:varzkExact} \text{Var}(Z_{k})&\sim\frac{\sigma^{2}\lambda_{k}^{-2\gamma}u_{k}^{6}(0)T^{2}}{2\mu_{k}^{5}(\theta)}+\frac{2\sigma^{4}\lambda_{k}^{-4\gamma}u_{k}^{4}(0)T^{3}}{3\mu_{k}^{5}(\theta)}+\frac{\sigma^{6}\lambda_{k}^{-6\gamma}u_{k}^{2}(0)T^{4}}{3\mu_{k}^{5}(\theta)}+\frac{\sigma^{8}\lambda_{k}^{-8\gamma}T^{5}}{15\mu_{k}^{5}(\theta)},\\
\label{eq:eakExact} \mathbb{E}(A_{k})&\sim\frac{\sigma^{2}\lambda_{k}^{-2\gamma}u_{k}^{2}(0)T}{\mu_{k}^{2}(\theta)}+\frac{\sigma^{4}\lambda_{k}^{-4\gamma}T^{2}}{2\mu_{k}^{2}(\theta)},\\
\label{eq:varakExact} \text{Var}(A_{k})&\sim\frac{\sigma^{2}\lambda_{k}^{-2\gamma}u_{k}^{6}(0)T^{2}}{2\mu_{k}^{3}(\theta)}+\frac{2\sigma^{4}\lambda_{k}^{-4\gamma}u_{k}^{4}(0)T^{3}}{3\mu_{k}^{3}(\theta)}+\frac{\sigma^{6}\lambda_{k}^{-6\gamma}u_{k}^{2}(0)T^{4}}{3\mu_{k}^{3}(\theta)}+\frac{\sigma^{8}\lambda_{k}^{-8\gamma}T^{5}}{15\mu_{k}^{3}(\theta)}.
\end{align}
These formulas will be used to obtain the exact asymptotic bias and the exact rate of convergence in the proof of asymptotic normality.
\end{remark}

\subsection{The Consistency of TFE}
In this subsection we prove the large-space consistency of the TFE $\widetilde{\theta}_{N}$, as $N\rightarrow\infty$. The proof relies on the classical version of the Strong Law of Large Numbers (cf. \cite[Theorem IV.3.2]{ShiryaevBookProbability}), which we state in the Appendix for sake of completeness. With this at hand, we now present the first main result of this paper.
\begin{theorem}[Consistency of TFE]\label{thm:Consistency}
 Let the assumptions \emph{(i) - (v)} be satisfied. Moreover, assume that
\begin{align}\label{eq:zk}
\sum_{k=1}^{\infty}\nu_{k}^{2}\, \mathbb{E}(Z_{k})=\infty.
\end{align}
Then,
\begin{align}
\lim_{N\rightarrow\infty}\widetilde{\theta}_{N}=\theta,\quad\mathbb{P}-\text{a.}\,\text{s.}.
\end{align}
\end{theorem}

\begin{proof}
By \eqref{eq:thetaDif},
\begin{align}\label{eq:thetaDif2}
\widetilde{\theta}_{N}-\theta=-\frac{\sum_{k=1}^{N}\nu_{k}A_{k}}{2\sum_{k=1}^{N}\nu_{k}^{2}\,\mathbb{E}(Z_{k})}\cdot\frac{\sum_{k=1}^{N}\nu_{k}^{2}\,\mathbb{E}(Z_{k})}{\sum_{k=1}^{N}\nu_{k}^{2}Z_{k}}.
\end{align}
We first study the second factor in \eqref{eq:thetaDif2}. Clearly,
\begin{align}
\sum_{N=1}^{\infty}\frac{\nu_{N}^{4}\,\text{Var}(Z_{N})}{\left(\sum_{k=1}^{N}\nu_{k}^{2}\,\mathbb{E}(Z_{k})\right)^{2}}&\leq\frac{\text{Var}(Z_{1})}{\left(\mathbb{E}(Z_{1})\right)^{2}}+\sum_{N=2}^{\infty}\frac{\nu_{N}^{4}\,\text{Var}(Z_{N})}{\sum_{k=1}^{N-1}\nu_{k}^{2}\,\mathbb{E}(Z_{k})\cdot\sum_{k=1}^{N}\nu_{k}^{2}\,\mathbb{E}Z_{k}}\\
\label{eq:DecompVarZNSqSumExpZk} &=\frac{\text{Var}(Z_{1})}{\left(\mathbb{E}(Z_{1})\right)^{2}}\!+\!\!\sum_{N=2}^{\infty}\!\frac{\nu_{N}^{2}\text{Var}(Z_{N})}{\mathbb{E}(Z_{N})}\!\left(\!\frac{1}{\sum_{k=1}^{N-1}\!\nu_{k}^{2}\mathbb{E}(Z_{k})}\!-\!\frac{1}{\sum_{k=1}^{N}\!\nu_{k}^{2}\mathbb{E}(Z_{k})}\!\right).
\end{align}
By \eqref{eq:ezk}, \eqref{eq:varzk} and the assumption (iii), as $N\rightarrow\infty$,
\begin{align}
\frac{\nu_{N}^{2}\,\text{Var}(Z_{N})}{\mathbb{E}(Z_{N})}=O\left(\nu_{N}^{2}\cdot\frac{\frac{\lambda_{N}^{-2\gamma}}{\mu_{N}^{5}(\theta)}\left(u_{N}^{2}(0)+\sigma^{2}T\lambda_{N}^{-2\gamma}\right)^{3}}{\frac{1}{\mu_{N}^{2}(\theta)}\left(u_{N}^{2}(0)+\sigma^{2}T\lambda_{N}^{-2\gamma}\right)^{2}}\right)=O\left(\frac{\lambda_{N}^{-2\gamma}}{\mu_{N}(\theta)}\left(u_{N}^{2}(0)+\sigma^{2}T\lambda_{N}^{-2\gamma}\right)\right).
\end{align}
Since $u_{0}\in H$, we have
\begin{align}\label{eq:LimitSquN}
\lim_{N\rightarrow\infty}u_{N}^{2}(0)=0.
\end{align}
Together with the assumptions (ii), (iv) and (v),
\begin{align}
\lim_{N\rightarrow\infty}\frac{\lambda_{N}^{-2\gamma}}{\mu_{N}(\theta)}\left(u_{N}^{2}(0)+\sigma^{2}T\lambda_{N}^{-2\gamma}\right)=\lim_{N\rightarrow\infty}\frac{1}{\mu_{N}(\theta)\nu_{N}^{\gamma/m}}\left(u_{N}^{2}(0)+\frac{\sigma^{2}T}{\nu_{N}^{\gamma/m}}\right)=0.
\end{align}
Hence, there exists a universal constant $C_{1}>0$ such that
\begin{align}
\frac{\nu_{N}^{2}\,\text{Var}(Z_{N})}{\mathbb{E}(Z_{N})}\leq C_{1}\quad\text{for all }\,N\in\bN,
\end{align}
which, together with \eqref{eq:DecompVarZNSqSumExpZk}, implies that
\begin{align}
\sum_{N=1}^{\infty}\frac{\nu_{N}^{4}\,\text{Var}(Z_{N})}{\left(\sum_{k=1}^{N}\nu_{k}^{2}\,\mathbb{E}(Z_{k})\right)^{2}}&\leq\frac{\text{Var}(Z_{1})}{\left(\mathbb{E}(Z_{1})\right)^{2}}+C_{1}\sum_{N=2}^{\infty}\left(\frac{1}{\sum_{k=1}^{N-1}\nu_{k}^{2}\,\mathbb{E}(Z_{k})}-\frac{1}{\sum_{k=1}^{N}\nu_{k}^{2}\,\mathbb{E}(Z_{k})}\right)\\
\label{eq:EstVarZNSqSumExpZk} &=\frac{\text{Var}(Z_{1})}{\left(\mathbb{E}(Z_{1})\right)^{2}}+\frac{C_{1}}{\nu_{1}^{2}\,\mathbb{E}(Z_{1})}<\infty.
\end{align}
Combining \eqref{eq:zk} with \eqref{eq:EstVarZNSqSumExpZk}, we conclude by Remark \ref{rem:SLLN} that
\begin{align}\label{eq:Consist1}
\lim_{N\rightarrow\infty}\frac{\sum_{k=1}^{N}\nu_{k}^{2}Z_{k}}{\sum_{k=1}^{N}\nu_{k}^{2}\,\mathbb{E}(Z_{k})}=1,\quad\mathbb{P}-\text{a.}\,\text{s.}.
\end{align}

Next, we will analyze the asymptotic behavior of the first factor in \eqref{eq:thetaDif2}. By \eqref{eq:ezk}, \eqref{eq:varak}, \eqref{eq:LimitSquN} and the assumptions (ii), (iv) and (v), as $N\rightarrow\infty$, we get that
\begin{align}
\frac{\text{Var}(A_{N})}{\mathbb{E}(Z_{N})}=O\left(\frac{\frac{\lambda_{N}^{-2\gamma}}{\mu_{N}^{3}(\theta)}\left(u_{N}^{2}(0)+\sigma^{2}T\lambda_{N}^{-2\gamma}\right)^{3}}{\frac{1}{\mu_{N}^{2}(\theta)}\left(u_{N}^{2}(0)+\sigma^{2}T\lambda_{N}^{-2\gamma}\right)^{2}}\right)=O\left(\frac{1}{\mu_{N}(\theta)\nu_{N}^{\gamma/m}}\left(u_{N}^{2}(0)+\frac{\sigma^{2}T}{\nu_{N}^{\gamma/m}}\right)\right)\rightarrow 0.
\end{align}
An argument similar to that of \eqref{eq:DecompVarZNSqSumExpZk} and \eqref{eq:EstVarZNSqSumExpZk} shows that, there exists a universal constant $C_{2}>0$, such that
\begin{align}
\sum_{N=1}^{\infty}\frac{\nu_{N}^{2}\text{Var}(A_{N})}{\left(\sum_{k=1}^{N}\nu_{k}^{2}\,\mathbb{E}(Z_{k})\right)^{2}}&\leq\frac{\text{Var}(A_{1})}{\nu_{1}^{2}\left(\mathbb{E}(Z_{1})\right)^{2}}+\sum_{N=2}^{\infty}\frac{\text{Var}(A_{N})}{\mathbb{E}(Z_{N})}\left(\frac{1}{\sum_{k=1}^{N-1}\nu_{k}^{2}\,\mathbb{E}(Z_{k})}-\frac{1}{\sum_{k=1}^{N}\nu_{k}^{2}\,\mathbb{E}(Z_{k})}\right)\\
\label{eq:EstVarANSqSumExpZk} &\leq\frac{\text{Var}(A_{1})}{\nu_{1}^{2}\left(\mathbb{E}(Z_{1})\right)^{2}}+\frac{C_{2}}{\nu_{1}^{2}\,\mathbb{E}(Z_{1})}<\infty.
\end{align}
In view of  Theorem \ref{thm:SLLN}, \eqref{eq:zk} and \eqref{eq:EstVarANSqSumExpZk} imply that
\begin{align}
\lim_{N\rightarrow\infty}\frac{\sum_{k=1}^{N}\nu_{k}\left(A_{k}-\mathbb{E}(A_{k})\right)}{\sum_{k=1}^{N}\nu_{k}^{2}\,\mathbb{E}(Z_{k})}=0,\quad\mathbb{P}-\text{a.}\,\text{s.}.
\end{align}
If the series $\sum_{k=1}^{N}\nu_{k}\,\mathbb{E}(A_{k})$ converges,  then by \eqref{eq:zk}, we clearly have that
\begin{align}\label{eq:LimitSumExpAkSumExpZk}
\lim_{N\rightarrow\infty}\frac{\sum_{k=1}^{N}\nu_{k}\,\mathbb{E}(A_{k})}{\sum_{k=1}^{N}\nu_{k}^{2}\,\mathbb{E}(Z_{k})}=0,
\end{align}
On the other hand, if the series in the numerator of \eqref{eq:LimitSumExpAkSumExpZk} diverges, then by Stolz--Ces\`{a}ro Theorem
\begin{align}
\lim_{N\rightarrow\infty}\frac{\sum_{k=1}^{N}\nu_{k}\,\mathbb{E}(A_{k})}{\sum_{k=1}^{N}\nu_{k}^{2}\,\mathbb{E}(Z_{k})}=\lim_{N\rightarrow\infty}\frac{\nu_{N}\,\mathbb{E}(A_{N})}{\nu_{N}^{2}\,\mathbb{E}(Z_{N})},
\end{align}
and by \eqref{eq:ezk}, \eqref{eq:eak} and the assumption (iv), as $N\rightarrow\infty$, we deduce that
\begin{align}
\frac{\mathbb{E}(A_{N})}{\nu_{N}\,\mathbb{E}(Z_{N})}=O\left(\frac{\frac{\lambda_{N}^{-2\gamma}}{\mu_{N}^{2}(\theta)}\left(u_{N}^{2}(0)+\sigma^{2}T\lambda_{N}^{-2\gamma}\right)}{\frac{\nu_{N}}{\mu_{N}^{2}(\theta)}\left(u_{N}^{2}(0)+\sigma^{2}T\lambda_{N}^{-2\gamma}\right)^{2}}\right)=O\left(\frac{1}{\nu_{N}\left(u_{N}^{2}(0)\lambda_{N}^{2\gamma}+\sigma^{2}T\right)}\right)=O\left(\nu_{N}^{-1}\right)\rightarrow 0.
\end{align}
Combining the above, we conclude that
\begin{align}\label{eq:Consist2}
\lim_{N\rightarrow\infty}\frac{\sum_{k=1}^{N}\nu_{k}A_{k}}{\sum_{k=1}^{N}\nu_{k}^{2}\,\mathbb{E}(Z_{k})}=0,\quad\mathbb{P}-\text{a.}\,\text{s.}.
\end{align}
Finally, by \eqref{eq:thetaDif2}, \eqref{eq:Consist1} and \eqref{eq:Consist2} we conclude the proof.
\end{proof}

\begin{remark} A note on condition \eqref{eq:zk} is in order. The divergence of the series \eqref{eq:zk} is needed for the Law of Large Numbers to hold true. In view of \eqref{eq:ezk}, the condition \eqref{eq:zk} can be equivalently stated in terms of the known primary objects -- the initial data $u(0)$, the asymptotics  of the eigenvalues of $\cA_{0}$ and $\cA_{1}$, and the covariance structure of the noise (see Proposition \ref{prop:EquiDivConds} below). In particular, the consistency of the TFE does not depend on the regularity of the initial data.
\end{remark}

\subsection{The Asymptotic Normality of TFE}
In this subsection, we will investigate the asymptotic normality of the TFE $\widetilde{\theta}_{N}$. The proof is based on classical  Central Limit Theorem (CLT) for independent random variables with the Lyapunov condition (which is a sufficient condition for the Lindeberg condition to hold). For convenience, we list this result in the Appendix, and the complete proof can be found, for instance, in~\cite[Section III.4]{ShiryaevBookProbability}.

In what follows we will make use of the following technical lemma; the proof is deferred to the Appendix.
\begin{lemma}\label{lemma:MomentEstxi}
Let $\xi_{k}(t)$, $k\in\bN, \ t\in[0,T]$, be defined as in \eqref{eq:notation}. Then, for any $n\in\mathbb{N}$, there exist a constant $D_{n}=D_{n}(t)>0$, depending only on $n$ and $t$, such that, for every $k\in\mathbb{N}$,
\begin{align}
\mathbb{E}\left(\xi_{k}^{n}(t)\right)\leq D_{n}\left(\frac{u_{k}^{2}(0)+\sigma^{2}t\,\lambda_{k}^{-2\gamma}}{\mu_{k}(\theta)}\right)^{n},\quad t\in[0,T].
\end{align}
\end{lemma}

Now we present a version of the large-space asymptotic normality of the TFE $\widetilde{\theta}_{N}$.

\begin{theorem}[Asymptotic Normality of TFE]\label{thm:AsymNormal}
In addition to the conditions of Theorem \ref{thm:Consistency}, assume further that
\begin{align}\label{eq:SumVarAk}
\sum_{k=1}^{\infty}\nu_{k}^{2}\,\emph{Var}\left(A_{k}\right)=\infty.
\end{align}
Then, as $N\rightarrow\infty$,
\begin{align}
\frac{\widetilde{\theta}_{N}-\theta+a_{N}}{b_{N}}\covdist\mathcal{N}(0,1),
\end{align}
where
\begin{align}
a_{N}:=\frac{\sum_{k=1}^{N}\nu_{k}\,\mathbb{E}(A_{k})}{2\sum_{k=1}^{N}\nu_{k}^{2}\,\mathbb{E}(Z_{k})},\quad b_{N}:=\frac{\sqrt{\sum_{k=1}^{N}\nu_{k}^{2}\,\emph{Var}(A_{k})}}{2\sum_{k=1}^{N}\nu_{k}^{2}\,\mathbb{E}(Z_{k})},
\end{align}
and where $\covdist$ denotes the convergence in distribution.
\end{theorem}

\begin{proof} The proof is split in two steps.

\medskip
\noindent
\textit{Step 1.} We will first show that the sequence $\set{\nu_{k}A_{k}}_{k\in\bN}$ satisfies the Lyapunov condition \eqref{eq:Lyapunov} with $\delta=2$. Clearly,
\begin{align}
\mathbb{E}\left(\left(A_{k}-\mathbb{E}(A_{k})\right)^{4}\right)&=\mathbb{E}\left(A_{k}^{4}\right)-4\,\mathbb{E}\left(A_{k}^{3}\right)\mathbb{E}(A_{k})+6\,\mathbb{E}\left(A_{k}^{2}\right)(\mathbb{E}(A_{k}))^{2}-3\left(\mathbb{E}(A_{k})\right)^{4}\\
&=\mathbb{E}\left(A_{k}^{4}\right)-4\,\mathbb{E}\left(A_{k}^{3}\right)\mathbb{E}(A_{k})+6\,\text{Var}(A_{k})\left(\mathbb{E}(A_{k})\right)^{2}+3\left(\mathbb{E}(A_{k})\right)^{4}.
\label{eq:fourthPower}
\end{align}
We will estimate each term in the above expression separately. To begin with, for every $k\in\mathbb{N}$, let $\zeta_{k}:=(\zeta_{k}(t))_{t\in[0,T]}$, where
\begin{align}
\zeta_{k}(t):=\int_{0}^{t}u_{k}(s)\,dw_{k}(s),\quad t\in[0,T].
\end{align}
By \eqref{eq:uk2} and \eqref{eq:notation}, for any $k\in\mathbb{N}$ and $t\in[0,T]$,
\begin{align}
u_{k}^{2}(t)=u_{k}^{2}(0)+\sigma^{2}\lambda_{k}^{-2\gamma}t-2\mu_{k}(\theta)\xi_{k}(t)+2\sigma\lambda_{k}^{-\gamma}\zeta_{k}(t),
\end{align}
which, when multiplied by $\xi_{k}(t)$, and then integrated on $[0,T]$, leads to
\begin{align}
A_{k}=2\sigma\lambda_{k}^{-\gamma}\int_{0}^{T}\zeta_{k}(t)\xi_{k}(t)\,dt,\quad k\in\mathbb{N}.
\end{align}
Hence, by the Cauchy--Schwartz inequality and the definition of $\xi_{k}$, for any $k\in\mathbb{N}$,
\begin{align}
\mathbb{E}\left(A_{k}^{4}\right)&\leq 16\sigma^{4}\lambda_{k}^{-4\gamma}\,\mathbb{E}\left(\left(\int_{0}^{T}\zeta_{k}^{2}(t)\,dt\cdot\int_{0}^{T}\xi_{k}^{2}(t)\,dt\right)^{2}\right)\\
&\leq 16\sigma^{4}\lambda_{k}^{-4\gamma}\left(\mathbb{E}\left(\left(\int_{0}^{T}\zeta_{k}^{2}(t)dt\right)^{4}\right)\mathbb{E}\left(\left(\int_{0}^{T}\xi_{k}^{2}(t)dt\right)^{4}\right)\right)^{1/2}\\
&\leq 16\sigma^{4}\lambda_{k}^{-4\gamma}\left(T^{2}\,\mathbb{E}\left(\left(\int_{0}^{T}\zeta_{k}^{4}(t)\,dt\right)^{2}\right)\cdot T^{4}\,\mathbb{E}\left(\xi_{k}^{8}\right)\right)^{1/2}\\
&\leq 16\sigma^{4}T^{7/2}\lambda_{k}^{-4\gamma}\left(\mathbb{E}\left(\int_{0}^{T}\zeta_{k}^{8}(t)\,dt\right)\mathbb{E}\left(\xi_{k}^{8}\right)\right)^{1/2}.
\end{align}
Moreover, by the Burkholder--Davis--Gundy inequality, there exists a constant $C_{1}=C_{1}(T)>0$, depending only on $T$, such that
\begin{align}
\mathbb{E}\left(\sup_{t\in[0,T]}\zeta_{k}^{8}(t)\right)\leq C_{1}\,\mathbb{E}\left(\left[\zeta_{k},\zeta_{k}\right]^{4}(T)\right)=C_{1}\,\mathbb{E}\left(\xi_{k}^{4}\right).
\end{align}
Together with Lemma \ref{lemma:MomentEstxi}, we obtain that, for any $k\in\mathbb{N}$,
\begin{align}\label{eq:EstAk4Moment}
\mathbb{E}\left(A_{k}^{4}\right)\leq 16\,C_{1}\sigma^{4}\, T^{4}\lambda_{k}^{-4\gamma}\sqrt{\mathbb{E}\left(\xi_{k}^{4}\right)\mathbb{E}\left(\xi_{k}^{8}\right)}\leq C_{2}\,\frac{\left(u_{k}^{2}(0)+\sigma^{2}T\lambda_{k}^{-2\gamma}\right)^{8}}{\mu_{k}^{6}(\theta)},
\end{align}
where $C_{2}:=16\,C_{1}\sqrt{D_{4}D_{8}}T^{2}>0$ is a constant depending only on $T$.

Next, we will study the last three terms of \eqref{eq:fourthPower}. In view of \eqref{eq:eak} and \eqref{eq:varak}, there exists a universal constant $C_{3}>0$, such that for any $k\in\mathbb{N}$,
\begin{align}\label{eq:EstAkVarAk}
\mathbb{E}(A_{k})\leq C_{3}\,\frac{\left(u_{k}^{2}(0)+\sigma^{2}T\lambda_{k}^{-2\gamma}\right)^{2}}{\mu_{k}^{2}(\theta)},\quad\text{Var}(A_{k})\leq C_{2}\,\frac{\left(u_{k}^{2}(0)+\sigma^{2}T\lambda_{k}^{-2\gamma}\right)^{4}}{\mu_{k}^{3}(\theta)}.
\end{align}
Hence, it suffices to estimate $\mathbb{E}(A_{k}^{3})$. By the definition of $A_{k}$ in \eqref{eq:thetaDif},
\begin{equation}
-u_{k}^{2}(0)Y_{k}-\sigma^{2}\lambda_{k}^{-2\gamma}X_{k}\leq A_{k}\leq\frac{1}{2}\xi_{k}^{2}+2\mu_{k}(\theta)Z_{k}.
\end{equation}
Moreover, since $\xi_{k}(t)$ is increasing in $t$, we deduce that
\begin{align}
Y_{k}=\int_{0}^{T}\xi_{k}(t)\,dt\leq T\xi_{k},\quad X_{k}=\int_{0}^{T}t\xi_{k}(t)\,dt\leq T^{2}\xi_{k},\quad Z_{k}=\int_{0}^{T}\xi_{k}^{2}(t)\,dt\leq T\xi_{k}^{2},
\end{align}
and thus,
\begin{align}
-T\left(u_{k}^{2}(0)+\sigma^{2}T\lambda_{k}^{-2\gamma}\right)\xi_{k}\leq A_{k}\leq\left(\frac{1}{2}+2\mu_{k}(\theta)T\right)\xi_{k}^{2}.
\end{align}
Together with Lemma \ref{lemma:MomentEstxi}, we obtain that
\begin{align}\label{eq:EstAk3Moment}
- D_{3}T\frac{\left(u_{k}^{2}(0)+\sigma^{2}T\lambda_{k}^{-2\gamma}\right)^{6}}{\mu_{k}^{3}(\theta)}\leq\mathbb{E}(A_{k}^{3})\leq D_{6}(2T+1)\frac{\left(u_{k}^{2}(0)+\sigma^{2}T\lambda_{k}^{-2\gamma}\right)^{6}}{\mu_{k}^{3}(\theta)}.
\end{align}

Combining \eqref{eq:fourthPower}--\eqref{eq:EstAk3Moment}, we conclude that there exists a constant $C_{4}=C_{4}(T)>0$, depending only on $T$, such that for any $k\in\mathbb{N}$,
\begin{align}
\mathbb{E}\left(\left(\nu_{k}A_{k}-\mathbb{E}(\nu_{k}A_{k})\right)^{4}\right)\leq C_{4}\,\frac{\nu_{k}^{4}}{\mu_{k}^{5}(\theta)}\left(u_{k}^{2}(0)+\sigma^{2}T\lambda_{k}^{-2\gamma}\right)^{8}.
\end{align}
On the other hand, by \eqref{eq:varak} again, we can find a universal constant $C_{5}>0$, such that
\begin{align}
\text{Var}\left(\nu_{k}A_{k}\right)\geq C_{5}\frac{\nu_{k}^{2}}{\mu_{k}^{3}(\theta)}\left(u_{k}^{2}(0)+\sigma^{2}T\lambda_{k}^{-2\gamma}\right)^{4},
\quad \text{for all }\,k\in\mathbb{N}.
\end{align}
In  view of the assumptions (ii)-(v), and since $\lim_{k\rightarrow\infty}u_{k}^{2}(0)=0$, we deduce that there exists a constant $C_{6}=C_{6}(c_{0},\sigma,T)>0$, depending on $c_{0}$, $\sigma$ and $T$, such that
\begin{align}
\mathbb{E}\left(\left(\nu_{k}A_{k}-\mathbb{E}(\nu_{k}A_{k})\right)^{4}\right)\leq C_{6}\text{Var}\left(\nu_{k}A_{k}\right),\quad\text{for all }\,k\in\mathbb{N}.
\end{align}
Finally, by \eqref{eq:SumVarAk}, we obtain that
\begin{align}
\lim_{N\rightarrow\infty}\frac{\sum_{k=1}^{N}\mathbb{E}\left(\left(\nu_{k}A_{k}-\mathbb{E}(\nu_{k}A_{k})\right)^{4}\right)}{\left(\sum_{k=1}^{N}\text{Var}\left(\nu_{k}A_{k}\right)\right)^{2}}\leq\lim_{N\rightarrow\infty}\frac{C_{6}}{\sum_{k=1}^{N}\text{Var}\left(\nu_{k}A_{k}\right)}=0.
\end{align}

\smallskip
\noindent
\textit{Step 2:} Note that
\begin{align}\label{eq:expandofAll}
\frac{\widetilde{\theta}_{N}-\theta+a_{N}}{b_{N}}=-\frac{\sum_{k=1}^{N}\nu_{k}\left(A_{k}-\mathbb{E}(A_{k})\right)}{2\,b_{N}\sum_{k=1}^{N}\nu_{k}^{2}Z_{k}}-\frac{\sum_{k=1}^{N}\nu_{k}\,\mathbb{E}(A_{k})}{2\,b_{N}\sum_{k=1}^{N}\nu_{k}^{2}Z_{k}}+\frac{a_{N}}{b_{N}}.
\end{align}
For the first term in \eqref{eq:expandofAll}, by \eqref{eq:Consist1}, Step 1 and Theorem \ref{thm:LyapunovCLT}, as $N\rightarrow\infty$, we get
\begin{align}\label{eq:CLTLyapunov}
-\frac{\sum_{k=1}^{N}\nu_{k}\left(A_{k}-\mathbb{E}(A_{k})\right)}{2\,b_{N}\sum_{k=1}^{N}\nu_{k}^{2}Z_{k}}=-\frac{\sum_{k=1}^{N}\nu_{k}^{2}\,\mathbb{E}(Z_{k})}{\sum_{k=1}^{N}\nu_{k}^{2}Z_{k}}\cdot\frac{\sum_{k=1}^{N}\nu_{k}\left(A_{k}-\mathbb{E}(A_{k})\right)}{\sqrt{\sum_{k=1}^{N}\nu_{k}^{2}\,\text{Var}(A_{k})}}\covdist\mathcal{N}(0,1).
\end{align}
Moreover, for the last two terms in \eqref{eq:expandofAll}, we derive that
\begin{align}
\frac{a_{N}}{b_{N}}-\frac{\sum_{k=1}^{N}\nu_{k}\,\mathbb{E}(A_{k})}{2\,b_{N}\sum_{k=1}^{N}\nu_{k}^{2}Z_{k}}&=\frac{2\sum_{k=1}^{N}\nu_{k}^{2}\,\mathbb{E}(Z_{k})}{\sqrt{\sum_{k=1}^{N}\nu_{k}^{2}\,\text{Var}(A_{k})}}\left(\frac{\sum_{k=1}^{N}\nu_{k}\,\mathbb{E}(A_{k})}{2\sum_{k=1}^{N}\nu_{k}^{2}\,\mathbb{E}(Z_{k})}-\frac{\sum_{k=1}^{N}\nu_{k}\,\mathbb{E}(A_{k})}{2\sum_{k=1}^{N}\nu_{k}^{2}Z_{k}}\right)\\
\label{eq:expandOfTail} &=\frac{\sum_{k=1}^{N}\nu_{k}^{2}\,\mathbb{E}(Z_{k})}{\sum_{k=1}^{N}\nu_{k}^{2}Z_{k}}\cdot\frac{\sum_{k=1}^{N}\nu_{k}\,\mathbb{E}(A_{k})}{\sqrt{\sum_{k=1}^{N}\nu_{k}^{2}\,\text{Var}(A_{k})}}\cdot\frac{\sum_{k=1}^{N}\nu_{k}^{2}\left(Z_{k}-\mathbb{E}(Z_{k})\right)}{\sum_{k=1}^{N}\nu_{k}^{2}\,\mathbb{E}(Z_{k})}.
\end{align}
In light of \eqref{eq:Consist1}, we only need to show that the product of the last two factors above converges to zero in probability, as $N\rightarrow\infty$. Note that, by the independence of $Z_{k}$, $k\in\mathbb{N}$,
\begin{align}
\mathbb{E}\!\left(\!\left(\frac{\sum_{k=1}^{N}\nu_{k}\,\mathbb{E}(A_{k})}{\sqrt{\sum_{k=1}^{N}\nu_{k}^{2}\,\text{Var}(A_{k})}}\cdot\frac{\sum_{k=1}^{N}\nu_{k}^{2}\left(Z_{k}-\mathbb{E}(Z_{k})\right)}{\sum_{k=1}^{N}\nu_{k}^{2}\,\mathbb{E}(Z_{k})}\right)^{2}\right)=\left(\frac{\sum_{k=1}^{N}\nu_{k}\,\mathbb{E}(A_{k})}{\sum_{k=1}^{N}\nu_{k}^{2}\,\mathbb{E}(Z_{k})}\right)^{2}\frac{\sum_{k=1}^{N}\nu_{k}^{4}\,\text{Var}(Z_{k})}{\sum_{k=1}^{N}\nu_{k}^{2}\,\text{Var}(A_{k})}.
\end{align}
By \eqref{eq:varzk}, \eqref{eq:varak} and the assumption (iii), there exists a universal constant $C_{7}>0$, such that
\begin{align}
\frac{\sum_{k=1}^{N}\nu_{k}^{4}\,\text{Var}(Z_{k})}{\sum_{k=1}^{N}\nu_{k}^{2}\,\text{Var}(A_{k})}&\leq C_{7}\frac{\sum_{k=1}^{N}\frac{\nu_{k}^{4}\lambda_{k}^{-2\gamma}}{\mu_{k}^{5}(\theta)}(u_{k}^{2}(0)+\sigma^{2}T\lambda_{k}^{-2\gamma})^3}{\sum_{k=1}^{N}\frac{\nu_{k}^{2}\lambda_{k}^{-2\gamma}}{\mu_{k}^{3}(\theta)}(u_{k}^{2}(0)+\sigma^{2}T\lambda_{k}^{-2\gamma})^3}\leq C_{7}c_{0},
\end{align}
Similarly, by \eqref{eq:ezk}, \eqref{eq:eak} and the assumption (iii),
\begin{align}\label{eq:RatioSumAkSumZk}
\frac{\sum_{k=1}^{N}\nu_{k}\,\mathbb{E}(A_{k})}{\sum_{k=1}^{N}\nu_{k}^{2}\,\mathbb{E}(Z_{k})}&\leq C_{8}\frac{\sum_{k=1}^{N}\frac{\nu_{k}\lambda_{k}^{-2\gamma}}{\mu_{k}^{2}(\theta)}(u_{k}^{2}(0)+\sigma^{2}T\lambda_{k}^{-2\gamma})}{\sum_{k=1}^{N}\frac{\nu_{k}^{2}}{\mu_{k}^{2}(\theta)}(u_{k}^{2}(0)+\sigma^{2}T\lambda_{k}^{-2\gamma})^{2}},
\end{align}
where $C_{8}>0$ is some universal constant. Using \eqref{eq:ezk} and \eqref{eq:zk}, we conclude that the series in the denominator on the right-hand side of \eqref{eq:RatioSumAkSumZk} diverges. Hence, the right-hand side of \eqref{eq:RatioSumAkSumZk} converges to 0, as $N\rightarrow\infty$, if the series in the numerator on the right-hand side of \eqref{eq:RatioSumAkSumZk} converges. Now assume that the numerator on the right-hand side of \eqref{eq:RatioSumAkSumZk} diverges. By Stolz--Ces\`{a}ro Theorem,
\begin{align}
\lim_{N\rightarrow\infty}\frac{\sum_{k=1}^{N}\frac{\nu_{k}\lambda_{k}^{-2\gamma}}{\mu_{k}^{2}(\theta)}\left(u_{k}^{2}(0)+\sigma^{2}T\lambda_{k}^{-2\gamma}\right)}{\sum_{k=1}^{N}\frac{\nu_{k}^{2}}{\mu_{k}^{2}(\theta)}\left(u_{k}^{2}(0)+\sigma^{2}T\lambda_{k}^{-2\gamma}\right)^{2}}=\lim_{N\rightarrow\infty}\frac{\frac{\nu_{N}\lambda_{N}^{-2\gamma}}{\mu_{N}^{2}(\theta)}\left(u_{N}^{2}(0)+\sigma^{2}T\lambda_{N}^{-2\gamma}\right)}{\frac{\nu_{N}^{2}}{\mu_{N}^{2}(\theta)}\left(u_{N}^{2}(0)+\sigma^{2}T\lambda_{N}^{-2\gamma}\right)^{2}}\leq\lim_{N\rightarrow\infty}\frac1{\sigma^{2}T\nu_{N}}=0.
\end{align}
Therefore, we have shown that
\begin{align}\label{eq:L2CovCLT}
\frac{\sum_{k=1}^{N}\nu_{k}\,\mathbb{E}(A_{k})}{\sqrt{\sum_{k=1}^{N}\nu_{k}^{2}\,\text{Var}(A_{k})}}\cdot\frac{\sum_{k=1}^{N}\nu_{k}^{2}\left(Z_{k}-\mathbb{E}(Z_{k})\right)}{\sum_{k=1}^{N}\nu_{k}^{2}\,\mathbb{E}(Z_{k})}\rightarrow 0\quad\text{in }\,L^{2}(\Omega),\quad N\rightarrow\infty.
\end{align}
Combining \eqref{eq:Consist1}, \eqref{eq:expandofAll}, \eqref{eq:expandOfTail} and \eqref{eq:L2CovCLT} completes the proof.
\end{proof}

 The next result provides some equivalent conditions, given explicitly in terms of the model coefficients, for \eqref{eq:zk} and \eqref{eq:SumVarAk} to hold. In particular, we note that the consistency and the asymptotic normality of the TFE do not depend on the regularity of the initial data.
\begin{proposition}\label{prop:EquiDivConds}
Under the assumptions \emph{(i) - (v)},
\begin{align}\label{eq:EquiCondExpZk} \sum_{k=1}^{\infty}\nu_{k}^{2}\,\mathbb{E}(Z_{k})=\infty\quad&\Leftrightarrow\quad\sum_{k=}^{\infty}\frac{\nu_{k}^{2}\lambda_{k}^{-4\gamma}}{\mu_{k}^{2}(\theta)}=\infty,\\
\label{eq:EquiCondVarAk}
\sum_{k=1}^{\infty}\nu_{k}^{2}\,\emph{Var}(A_{k})=\infty\quad&\Leftrightarrow\quad\sum_{k=1}^{\infty}\frac{\nu_{k}^{2}\lambda_{k}^{-8\gamma}}{\mu_{k}^{3}(\theta)}=\infty.
\end{align}
\end{proposition}

\begin{proof}
We will only present the proof for \eqref{eq:EquiCondVarAk}, as \eqref{eq:EquiCondExpZk} can be obtained similarly. Clearly, \eqref{eq:varak} implies the ``$\Leftarrow$" direction in \eqref{eq:EquiCondVarAk}. Now assume that
\begin{align}
\sum_{k=1}^{\infty}\nu_{k}^{2}\,\text{Var}(A_{k})=\infty,
\end{align}
which, by \eqref{eq:varak}, is equivalent to
\begin{align}
\infty=\sum_{k=1}^{\infty}\frac{\nu_{k}^{2}\lambda_{k}^{-2\gamma}}{\mu_{k}^{3}(\theta)}\left(u_{k}^{2}(0)+\sigma^{2}T\lambda_{k}^{-2\gamma}\right)^{3}&=\sum_{k=1}^{\infty}\frac{\nu_{k}^{2}\lambda_{k}^{-2\gamma}u_{k}^{6}(0)}{\mu_{k}^{3}(\theta)}+3\sigma^{2}T\sum_{k=1}^{\infty}\frac{\nu_{k}^{2}\lambda_{k}^{-4\gamma}u_{k}^{4}(0)}{\mu_{k}^{3}(\theta)}\\
&\quad\,\,+3\sigma^{4}T^{2}\sum_{k=1}^{\infty}\frac{\nu_{k}^{2}\lambda_{k}^{-6\gamma}u_{k}^{2}(0)}{\mu_{k}^{3}(\theta)}+\sigma^{6}T^{3}\sum_{k=1}^{\infty}\frac{\nu_{k}^{2}\lambda_{k}^{-8\gamma}}{\mu_{k}^{3}(\theta)}.
\end{align}
Hence, it suffices to show that the first three series on the right-hand side above are all convergent. We will only check the first series, and the other two can be verified using a similar argument. Indeed, by the assumptions (ii) - (v) and since $u(0)\in H$ (so that $\lim_{k\rightarrow\infty}u_{k}(0)=0$), there exists a universal constant $C>0$ such that
\begin{align}
\sum_{k=1}^{\infty}\frac{\nu_{k}^{2}\lambda_{k}^{-2\gamma}u_{k}^{6}(0)}{\mu_{k}^{3}(\theta)}\leq c_{0}^{2}\,C\sum_{k=1}^{\infty}u_{k}^{2}(0)<\infty.
\end{align}
This concludes the proof.
\end{proof}

 We conclude this section by providing the asymptotically equivalent formulas for the sequences $\{a_{N}\}_{N\in\mathbb{N}}$ and $\{b_{N}\}_{N\in\mathbb{N}}$ in Theorem~\ref{thm:AsymNormal}, given in terms of the model coefficients. In light of Proposition~\ref{prop:EquiDivConds}, the relations \eqref{eq:zk} and \eqref{eq:SumVarAk} imply that each of the last terms in \eqref{eq:ezkExact}--\eqref{eq:varakExact} give the exact leading order term for $\mathbb{E}(Z_{k})$, $\text{Var}(Z_{k})$, $\mathbb{E}(A_{k})$ and $\text{Var}(A_{k})$, respectively. The following result follows immediately from Stolz--Ces\`{a}ro Theorem.
\begin{corollary}
Under the conditions of Theorem \ref{thm:AsymNormal}, as $N\rightarrow\infty$, we have
\begin{align}
a_{N}\sim \frac3{T}\frac{\sum_{k=1}^{N}\frac{\nu_{k}\lambda_{k}^{-4\gamma}}{\mu_{k}^{2}(\theta)}}{\sum_{k=1}^{N}\frac{\nu_{k}^{2}\lambda_{k}^{-4\gamma}}{\mu_{k}^{2}(\theta)}},\quad b_{N}\sim \sqrt{\frac{12}{5T}}\frac{\sqrt{\sum_{k=1}^{N}\frac{\nu_{k}^{2}\lambda_{k}^{-8\gamma}}{\mu_{k}^{3}(\theta)}}}{\sum_{k=1}^{N}\frac{\nu_{k}^{2}\lambda_{k}^{-4\gamma}}{\mu_{k}^{2}(\theta)}}.
\end{align}
\end{corollary}

\section{Examples}\label{sec:Examples}

In this section, we will present two examples that illustrate the theoretical results obtained in Section~\ref{sec:mainResults}. Throughout this section, let $G\subseteq\mathbb{R}^{d}$ be a smooth and bounded domain, $H:=L^{2}(G)$, and let $\Delta$ be the Laplace operator on $G$ with zero boundary condition. It is known (cf.~\cite{Shubin}) that $\Delta$ has a complete orthonormal system of eigenfunctions $\{h_{k}\}_{k\in\mathbb{N}}$ in $H$. Moreover, the corresponding eigenvalues $\{\tau_{k}\}_{k\in\mathbb{N}}$ can be arranged such that $0\leq -\tau_{1}\leq -\tau_{2}\leq\cdots$, and there exists a universal constant $c_{1}>0$ so that
\begin{align}
\lim_{k\rightarrow\infty}\left|\tau_{k}\right|\cdot k^{-2/d}=c_{1}.
\end{align}
\begin{example}\label{eg:FracStoHeatEq}
Consider the following fractional stochastic heat equation driven by an additive noise, possibly colored in space,
\begin{align}
du(t,x)+\theta(-\Delta)^{\beta}u(t,x)\,dt=\sigma\sum_{k=1}^{\infty}\lambda_{k}^{-\gamma}h_{k}(x)\,dw_{k}(t),\quad t\in[0,T],\quad x\in G,
\end{align}
with initial condition $u(0,x)=u_{0}(x)\in H$, where $\theta>0$, $\beta>0$, $\gamma\geq 0$ and $\sigma\in\mathbb{R}\setminus\{0\}$ are constants, and where $\lambda_{k}:=\sqrt{-\tau_{k}}$, $k\in\mathbb{N}$. In this case, $\rho_{k}=0$ for all $k\in\mathbb{N}$, and
\begin{align}
\nu_{k}\sim c_{1}k^{2\beta/d},\quad\mu_{k}(\theta)\sim c_{1}\theta\,k^{2\beta/d},\quad\lambda_{k}\sim\sqrt{c_{1}}\,k^{1/d},\quad k\rightarrow\infty.
\end{align}
Together with Proposition \ref{prop:EquiDivConds}, the conditions \eqref{eq:zk} and \eqref{eq:SumVarAk} are equivalent to
\begin{align}
\frac{1}{c_{1}^{2\gamma}\theta^{2}}\sum_{k=1}^{\infty}\frac{1}{k^{4\gamma/d}}
=\infty,
\quad\text{and}\quad\frac{1}{c_{1}^{1+4\gamma}\theta^{3}}\sum_{k=1}^{\infty}\frac{1}{k^{(2\beta+8\gamma)/d}}=\infty,
\end{align}
respectively. Therefore, the consistency and the asymptotic normality hold for the TFE $\tilde{\theta}_N$ given by \eqref{eq:thetaMain}, whenever
\begin{align}
2\beta+8\gamma\leq d.
\end{align}
\end{example}
\begin{example}\label{eg:DiagSPDE}
Let us consider the following SPDE, with the parameter of interest $\theta$ in front of lower order differential operator,
\begin{align}
du(t,x)+\left(\Delta u(t,x)+\theta u(t,x)\right)dt=\sum_{k=1}^{\infty}h_{k}(x)\,dw_{k}(t),\quad t\in[0,T],\quad x\in G,
\end{align}
with initial condition $u(0,x)=u_{0}(x)\in H$. In this case, $\gamma=0$, $\nu_{k}\equiv 1$ for all $k\in\mathbb{N}$, and
\begin{align}
\rho_{k}\sim c_{1}\,k^{2/d},\quad\mu_{k}(\theta)\sim\theta+c_{1}\,k^{2/d},\quad k\rightarrow\infty.
\end{align}
Together with Proposition \ref{prop:EquiDivConds}, the conditions \eqref{eq:zk} and \eqref{eq:SumVarAk} are equivalent to
\begin{align}
\sum_{k=1}^{\infty}\frac{1}{\left(\theta+c_{1}k^{2/d}\right)^{2}}=\infty\quad\text{and}\quad\sum_{k=1}^{\infty}\frac{1}{\left(\theta+c_{1}k^{2/d}\right)^{3}}=\infty,
\end{align}
respectively. Therefore, in order for  the consistency and the asymptotic normality of $\tilde{\theta}_N$ to hold true, we need to have $d\geq 6$.
\end{example}

{\small

\section{Appendix}

\noindent
\textit{Proof of Theorem \ref{prop:asymp}.} Due to the nature of desired asymptotic results, the underlying computations are somehow extensive and tedious. For simplicity of brevity, we will only provide the proof for the special case when $u_{0}=0$, $\gamma=0$ and $\sigma=1$, but the genera case can be verified using similar arguments and the details can be obtained from the authors upon request. Most of the evaluations were performed using symbolic computations in Mathematica. For each $k\in\mathbb{N}$, when $u_{0}=0$, $\gamma=0$ and $\sigma=1$,
\begin{align}
u_{k}(t)=e^{-\mu_{k}(\theta)t}\int_{0}^{t}e^{\mu_{k}(\theta)s}\,dw_{k}(s),\quad  t\in[0,T],
\end{align}
and with the notations introduced in \eqref{eq:notation} and \eqref{eq:thetaDif}, we get
\begin{align}
A_{k}=\frac{1}{2}\xi_{k}^{2}-X_{k}+2\mu_{k}(\theta)Z_{k}.
\end{align}
Note that for any $t\in[0,T]$, $u_{k}(t)$ is a centered normal random variable with variance  $(1-e^{-2\mu_{k}(\theta)t})/(2\mu_{k}(\theta))$, and thus,
\begin{align}\label{eq:CenNormalMoments}
\mathbb{E}\left(u_{k}^{2n}(t)\right)=(2n-1)!!\cdot\left(\frac{1-e^{-2\mu_{k}(\theta)t}}{2\mu_{k}(\theta)}\right)^{n},
\quad n\in\mathbb{N}.
\end{align}
We first verify \eqref{eq:ezk}, which now reduces to
\begin{align}\label{eq:newezk}
\mathbb{E}(Z_{k})\asymp\frac{T^{2}}{\mu_{k}^{2}(\theta)},\quad k\rightarrow\infty,
\end{align}
by computing
\begin{align}\label{eq:DecompEzk}
\mathbb{E}(Z_{k})=\mathbb{E}\left(\int_{0}^{T}\xi_{k}^{2}(t)dt\right)=\int_{0}^{T}\mathbb{E}\left(\xi_{k}^{2}(t)\right)dt
=2\int_{0}^{T}\int_{0}^{t}\mathbb{E}\left(\xi_{k}(s)u_{k}^{2}(s)\right)ds\,dt,
\end{align}
where the last equality follows from the definition of $\xi_{k}$ in \eqref{eq:notation}. By It\^{o}'s formula,
\begin{align}
d\xi_{k}(t)u_{k}^{2}(t)= \left(u_{k}^{4}(t)+\xi_{k}(t)\right)dt-2\mu_{k}(\theta)\xi_{k}(t)u_{k}^{2}(t)\,dt+2u_{k}(t)\xi_{k}(t)dw_{k}(t).
\end{align}
Taking the expectations on both sides above, using the definition of $\xi_{k}(t)$ in \eqref{eq:notation} and \eqref{eq:CenNormalMoments}, we obtain that, for $t\in[0,T]$,
\begin{align}\label{eq:Exikuk2}
\mathbb{E}\left(\xi_{k}(t)u_{k}^{2}(t)\right)=\frac{1-e^{-2\mu_{k}(\theta)t}}{8\mu_{k}^{3}(\theta)}-\frac{5te^{-2\mu_{k}(\theta)t}}{4\mu_{k}^{2}(\theta)}+\frac{3\left(1-e^{-2\mu_{k}(\theta)t}\right)e^{-2\mu_{k}(\theta)t}}{8\mu_{k}^{3}(\theta)}+\frac{t}{4\mu_{k}^{2}(\theta)}.
\end{align}
Therefore, by \eqref{eq:DecompEzk},
\begin{align}\label{eq:ExactEzk}
\mathbb{E}(Z_{k})=\frac{35-3e^{-4\mu_{k}(\theta)T}-32e^{-2\mu_{k}(\theta)T}}{64\mu_{k}^{5}(\theta)}-\frac{9T+10T\,e^{-2\mu_{k}(\theta)T}}{16\mu_{k}^{4}(\theta)}+\frac{T^{2}}{8\mu_{k}^{3}(\theta)}+\frac{T^{3}}{12\mu_{k}^{2}(\theta)},
\end{align}
which  leads to \eqref{eq:newezk}, since by the assumption (ii), the first three terms above all have higher orders than the last term, as $k\rightarrow\infty$, and since $T>0$ is a fixed constant.

Next, we study the asymptotic order of $\text{Var}(Z_{k})$, $k\rightarrow\infty$, given by \eqref{eq:varzk},  which now reduces to
\begin{align}\label{eq:NewVarzk}
\text{Var}(Z_{k})\asymp\frac{T^{3}}{\mu_{k}^{5}(\theta)},\quad k\rightarrow\infty.
\end{align}
In light of \eqref{eq:ExactEzk}, we are left to compute $\mathbb{E}(Z_{k}^{2})$. By It\^{o}'s formula,  and since the It\^{o} integral terms have zero expectation, we have
\begin{align}\label{eq:DecompEzk2}
\mathbb{E}(Z_{k}^{2})=2\!\int_{0}^{T}\!\mathbb{E}\!\left(Z_{k}(t)\xi_{k}^{2}(t)\right)dt=2\!\int_{0}^{T}\!\!\int_{0}^{t}\mathbb{E}\!\left(\xi_{k}^{4}(s)\right)\!ds\,dt+4\!\int_{0}^{T}\!\!\int_{0}^{t}\mathbb{E}\!\left(Z_{k}(s)\xi_{k}(s)u_{k}^{2}(s)\right)\!ds\,dt.
\end{align}
To compute the first expectation in \eqref{eq:DecompEzk2}, by It\^{o}'s formula again, we continue
\begin{align}
d\xi_{k}^{4}(t)&=4\xi_{k}^{3}(t)u_{k}^{2}(t)\,dt,\\
d\!\left(\xi_{k}^{3}(t)u_{k}^{2}(t)\right)&=\left(3\xi_{k}^{2}(t)u_{k}^{4}(t)+\xi_{k}^{3}(t)-2\mu_{k}(\theta)\xi_{k}^{3}(t)u_{k}^{2}(t)\right)dt+2\xi_{k}^{3}(t)u_{k}(t)\,dw_{k}(t),\\
d\xi_{k}^{3}(t)&=3\xi_{k}^{2}(t)u_{k}^{2}(t)\,dt,\\
d\!\left(\xi_{k}^{2}(t)u_{k}^{4}(t)\right)&=2\left(\xi_{k}(t)u_{k}^{6}(t)+3\xi_{k}^{2}(t)u_{k}^{2}(t)-2\mu_{k}(\theta)\xi_{k}^{2}(t)u_{k}^{4}(t)\right)dt+4\xi_{k}^{2}(t)u_{k}^{3}(t)\,dw_{k}(t).
\end{align}
Hence, we only need to compute $\mathbb{E}(\xi_{k}(s)u_{k}^{6}(s))$ and $\mathbb{E}(\xi_{k}^{2}(s)u_{k}^{2}(s))$. Again by It\^{o}'s formula, we get
\begin{align}
d\xi_{k}(t)u_{k}^{6}(t)&=\left(u_{k}^{8}(t)+15\xi_{k}(t)u_{k}^{4}(t)-6\mu_{k}(\theta)\xi_{k}(t)u_{k}^{6}(t)\right)dt+6\xi_{k}(t)u_{k}^{5}(t)\,dw_{k}(t),\\
d\xi_{k}^{2}(t)u_{k}^{2}(t)&=\left(2\xi_{k}(t)u_{k}^{4}(t)+\xi_{k}^{2}(t)-2\mu_{k}(\theta)\xi_{k}^{2}(t)u_{k}^{2}(t)\right)dt+2\xi_{k}^{2}(t)u_{k}(t)\,dw_{k}(t),\\
d\xi_{k}(t)u_{k}^{4}(t)&=\left(u_{k}^{6}(t)+6\xi_{k}(t)u_{k}^{2}(t)-4\mu_{k}(\theta)\xi_{k}(t)u_{k}^{4}(t)\right)dt+4\xi_{k}(t)u_{k}^{3}(t)\,dw_{k}(t).
\end{align}
Therefore, by \eqref{eq:CenNormalMoments} and \eqref{eq:Exikuk2}, we can obtain first $\mathbb{E}(\xi_{k}(s)u_{k}^{6}(s))$ and $\mathbb{E}(\xi_{k}^{2}(s)u_{k}^{2}(s))$, and then $\mathbb{E}(\xi_{k}^{3}(s))$, $\mathbb{E}(\xi_{k}^{2}(s)u_{k}^{4}(s))$ and $\mathbb{E}(\xi_{k}^{3}(s)u_{k}^{2}(s))$, and finally $\mathbb{E}(\xi_{k}^{4}(s))$. A similar argument leads to the computation of the second expectation in \eqref{eq:DecompEzk2}.

To sum up, with the help of Mathematica, we obtain that
\begin{align}
\text{Var}(Z_{k})&=-\frac{16917}{512\,\mu_{k}^{10}(\theta)}+\frac{3\,e^{-8\mu_{k}(\theta)T}}{128\,\mu_{k}^{10}(\theta)}+\frac{79\,e^{-6\mu_{k}(\theta)T}}{128\,\mu_{k}^{10}(\theta)}+\frac{2953\,e^{-4\mu_{k}(\theta)T}}{512\,\mu_{k}^{10}(\theta)}+\frac{3409\,e^{-2\mu_{k}(\theta)T}}{128\,\mu_{k}^{10}(\theta)}\\
&\quad\,\,+\frac{1093T}{32\,\mu_{k}^{9}(\theta)}+\frac{45T\,e^{-6\mu_{k}(\theta)T}}{64\,\mu_{k}^{9}(\theta)}+\frac{1165T\,e^{-4\mu_{k}(\theta)T}}{128\,\mu_{k}^{9}(\theta)}+\frac{2321T\,e^{-2\mu_{k}(\theta)T}}{64\,\mu_{k}^{9}(\theta)}\\
&\quad\,\,-\frac{659T^{2}}{64\,\mu_{k}^{8}(\theta)}+\frac{53T^{2}e^{-4\mu_{k}(\theta)T}}{16\,\mu_{k}^{8}(\theta)}+\frac{71T^{2}e^{-2\mu_{k}(\theta)T}}{8\,\mu_{k}^{8}(\theta)}-\frac{5T^{3}}{12\,\mu_{k}^{7}(\theta)}-\frac{5T^{3}e^{-4\mu_{k}(\theta)T}}{8\,\mu_{k}^{7}(\theta)}\\
&\quad\,\,-\frac{113T^{3}e^{-2\mu_{k}(\theta)T}}{24\,\mu_{k}^{7}(\theta)}+\frac{23T^{4}}{48\,\mu_{k}^{6}(\theta)}-\frac{5T^{4}e^{-2\mu_{k}(\theta)T}}{2\mu_{k}^{6}(\theta)}+\frac{T^{5}}{15\mu_{k}^{5}(\theta)},
\end{align}
which clearly implies \eqref{eq:NewVarzk}, and thus completes the proof.\hfill $\Box$

\medskip
\noindent
\textit{Proof of Lemma \ref{lemma:MomentEstxi}:} By \eqref{eq:FModeSols} and Cauchy--Schwartz inequality, for any $0\leq s\leq t\leq T$,
\begin{align}
u_{k}^{2}(s)&=e^{-2\mu_{k}(\theta)s}\left(u_{k}(0)+\sigma\lambda_{k}^{-\gamma}\int_{0}^{s}e^{\mu_{k}(\theta)r}\,dw_{k}(r)\right)^{2}\\
&\leq e^{-2\mu_{k}(\theta)s}\left(u_{k}^{2}(0)+\sigma^{2}\lambda_{k}^{-2\gamma}t\right)\left(1+\frac{1}{t}\left(\int_{0}^{s}e^{\mu_{k}(\theta)r}\,dw_{k}(r)\right)^{2}\right).
\end{align}
Hence, for any $t\in[0,T]$, and $n\in\mathbb{N}$,
\begin{align}
\xi_{k}^{n}(t)&\leq \left(u_{k}^{2}(0)+\sigma^{2}\lambda_{k}^{-2\gamma}t\right)^{n}\left\{\int_{0}^{t}e^{-2\mu_{k}(\theta)s}\left[1+\frac{1}{t}\left(\int_{0}^{s}e^{\mu_{k}(\theta)r}\,dw_{k}(r)\right)^{2}\right]ds\right\}^{n}\\
&=\left(u_{k}^{2}(0)+\sigma^{2}\lambda_{k}^{-2\gamma}t\right)^{n}\left[\frac{1-e^{-2\mu_{k}(\theta)t}}{2\mu_{k}(\theta)}+\frac{1}{t}\int_{0}^{t}e^{-2\mu_{k}(\theta)s}\left(\int_{0}^{s}e^{\mu_{k}(\theta)r}\,dw_{k}(r)\right)^{2}ds\right]^{n}\\
&\leq\left(u_{k}^{2}(0)+\sigma^{2}\lambda_{k}^{-2\gamma}t\right)^{n}\left\{\left(\frac{1-e^{-2\mu_{k}(\theta)t}}{\mu_{k}(\theta)}\right)^{n}\!\!+\frac{2^{n}}{t^{n}}\left[\int_{0}^{t}e^{-2\mu_{k}(\theta)s}\left(\int_{0}^{s}e^{\mu_{k}(\theta)r}\,dw_{k}(r)\right)^{2}\!ds\right]^{n}\right\}.
\end{align}
By~\cite[Theorem 2.1]{Lototsky2009Survey}, there exists a constant $\widetilde{D}_{n}=\widetilde{D}_{n}(t)>0$, such that
\begin{align}
\mathbb{E}\left(\left[\int_{0}^{t}e^{-2\mu_{k}(\theta)s}\left(\int_{0}^{s}e^{\mu_{k}(\theta)r}\,dw_{k}(r)\right)^{2}ds\right]^{n}\right)\leq\frac{\widetilde{D}_{n}}{\mu_{k}^{n}(\theta)}.
\end{align}
Therefore, for any $t\in[0,T]$ and $n\in\mathbb{N}$,
\begin{align}
\mathbb{E}\left(\xi_{k}^{n}(t)\right)\leq\left(u_{k}^{2}(0)+\sigma^{2}\lambda_{k}^{-2\gamma}t\right)^{n}\left(\frac{1}{\mu_{k}^{n}(\theta)}+\frac{2^{n}\widetilde{D}_{n}}{t^{n}}\frac1{\mu_{k}^{n}(\theta)}\right)=D_{n}\left(\frac{u_{k}^{2}(0)+\sigma\lambda_{k}^{-2\gamma}t}{\mu_{k}(\theta)}\right)^{n},
\end{align}
where $D_{n}=D_{n}(t):=1+2^{n}t^{-n}\widetilde{D}_{n}$.\hfill $\Box$

\medskip
We conclude the appendix by listing, for sake of completeness,  a version of law of large numbers and a version of central limit theorem used in the proofs of the main results in this paper.
For the detailed proofs of these results we refer the reader to \cite{ShiryaevBookProbability}.

\begin{theorem}\label{thm:SLLN} \textbf{\emph{(Strong Law of Large Number)}}$\,$
Let $\{\eta_{n}\}_{n\in\mathbb{N}}$ be a sequence of independent random variables, and let $\{b_{n}\}_{n\in\mathbb{N}}$ be a sequence of non-decreasing positive numbers such that $\lim_{n\rightarrow\infty}b_{n}=\infty$. If
\begin{align}
\sum_{n=1}^{\infty}\frac{\emph{Var}\left(\eta_{n}\right)}{b_{n}^{2}}<\infty,
\end{align}
then
\begin{align}
\lim_{n\rightarrow\infty}\frac{1}{b_{n}}\sum_{k=1}^{n}\left(\eta_{k}-\mathbb{E}(\eta_{k})\right)=0,\quad\mathbb{P}-\text{a.}\,\text{s.}.
\end{align}
\end{theorem}
\begin{remark}\label{rem:SLLN}
As an immediate corollary, if $\{\eta_{n}\}_{n\in\mathbb{N}}$ is a sequence of independent non-negative random variables with
\begin{align}
\sum_{n=1}^{\infty}\mathbb{E}(\eta_{n})=\infty\quad\text{and}\quad\sum_{n=1}^{\infty}\frac{\text{Var}\left(\eta_{n}\right)}{\left(\sum_{k=1}^{n}\mathbb{E}(\eta_{k})\right)^{2}}<\infty,
\end{align}
then
\begin{align}
\lim_{n\rightarrow\infty}\frac{\sum_{k=1}^{n}\eta_{k}}{\sum_{k=1}^{n}\mathbb{E}(\eta_{k})}=1,\quad\mathbb{P}-\text{a.}\,\text{s.}.
\end{align}
\end{remark}

\begin{theorem}\label{thm:LyapunovCLT} \textbf{\emph{(Lyapunov Central Limit Theorem)}}
Let $\{\eta_{n}\}_{n\in\mathbb{N}}$ be a sequence of independent random variables with finite second moments. If there exists some $\delta>0$, such that
\begin{align}\label{eq:Lyapunov}
\lim_{n\rightarrow\infty}\frac{1}{\left(\sum_{k=1}^{n}\emph{Var}(\eta_{k})\right)^{2+\delta}}\sum_{k=1}^{n}\mathbb{E}\left(\left|\eta_{k}-\mathbb{E}(\eta_{k})\right|^{2+\delta}\right)=0,
\end{align}
then
\begin{align}
\frac{\sum_{k=1}^{n}\left(\eta_{k}-\mathbb{E}(\eta_{k})\right)}{\sqrt{\sum_{k=1}^{n}\emph{Var}(\eta_{k})}}\covdist\mathcal{N}(0,1),\quad n\rightarrow\infty.
\end{align}
\end{theorem}

\section*{Acknowledgments}
Part of the research was performed while Igor Cialenco was visiting the Institute for Pure and Applied Mathematics (IPAM), which is supported by the National Science Foundation. The authors would like to thank the anonymous referees, the associate editor and the editor for their helpful comments and suggestions which improved greatly the final manuscript.

\bibliographystyle{alpha}

\def\cprime{$'$}

}
\end{document}